\documentclass[12pt,reqno]{article}%\amsart

\usepackage{amsmath,amssymb,amsthm}
\usepackage[dvips]{graphicx}
\usepackage{subfigure}
\usepackage{latexsym}
\usepackage{eucal}
\usepackage[latin1]{inputenc}
\usepackage[english,activeacute]{babel}

\newtheorem{theorem}{Theorem}[section]
\newtheorem{proposition}[theorem]{Proposition}
\newtheorem{lemma}[theorem]{Lemma}

\theoremstyle{definition}

\newtheorem{remark}[theorem]{Remark}
\newtheorem{definition}[theorem]{Definition}

\title{{\bf\Large Stability of standing waves for logarithmic Schr\"{o}dinger equation with attractive delta potential}}
\author{{\bf\large Jaime Angulo Pava}\footnote{Email: angulo@ime.usp.br}\hspace{3mm} {\bf\large and} {\bf\large Alex Hernandez Ardila }\footnote{Email: alexha@ime.usp.br}\hspace{2mm}
{\bf\large}\vspace{1mm}\\
{\it\small Department of Mathematics, IME-USP}\\
 {\it\small Rua do Mat\~ao 1010, Cidade Universit\'aria,}\\
{\it\small  CEP 05508-090, S\~ao Paulo, SP, Brazil}
\vspace{3mm}}

\date{}
\begin{document}
\maketitle
%%%%%%%%%%%%%% ABSTRACT %%%%%%%%%%%%%% ABSTRACT %%%%%%%%%%%%%%%%%%%%%%%%%%%%%%%%%%%%%%%%%%%%%%%%%%%%%%%%%%%%

\begin{abstract}
We consider the one-dimensional logarithmic Schr\"{o}dinger equation  with a  delta potential. Global well-posedness is verified for the Cauchy problem  in $H^{1}(\mathbb{R})$ and in an appropriate Orlicz space. In the attractive case, we prove orbital stability of the ground states via variational approach. \\
 

 \end{abstract}
 
%%%%%%%%%%%%% INTRODUCTION %%%%%%%%%%%%%%% INTRODUCTION %%%%%%%%%%%%%%%%%%%%%%%%%%%%%%%%%%%%%%%%%%

\section{Introduction}

The present paper is devoted to the analysis of existence and stability of the ground states for the following nonlinear  Schr\"{o}dinger equation  with a  delta potential: 
\begin{equation}
\label{00NL}
 i\partial_{t}u+(\partial^{2}_{x}+ \gamma\delta(x))u+u\, \mbox{Log}\left|u\right|^{2}=0, 
\end{equation}
where  $\gamma\in\mathbb{R}$ and $u=u(x,t)$ is a complex-valued function of $(x,t)\in \mathbb{R}\times\mathbb{R}$. Here, $\delta(x)$ is the delta measure at the origin. The parameter $\gamma$ is real; when positive, the potential is called attractive, otherwise repulsive. 

%The equation \eqref{00NL} can be viewed as a model of the  singular interaction  between nonlinear wave and an inhomogeneity; the delta potential is used to model an impurity, or defect, localized at the origin.

In the absence of the delta potential, the equation \eqref{00NL} has been proposed in order to obtain a nonlinear equation which helped to quantify departures from the strictly linear regime, preserving some aspects of quantum mechanics, such as separability and additivity of total energy for non-interacting subsystems, the validity of the lower energy bound and Planck's relation for all stationary states (see \cite{BLSE, CAS}). This equation admits  applications in quantum mechanics, quantum optics, nuclear physics, fluid dynamics, plasma physics and  Bose-Einstein condensation (see \cite{AH, APLES} and references therein). 

The formal expression $-\partial^{2}_{x}-\gamma\delta(x)$ which appears in \eqref{00NL} admits a precise interpretation as a self-adjoint  operator $H_{\gamma}$ on $L^{2}(\mathbb{R^{}}^{})$. Indeed, for $u$, $v\in H^{1}(\mathbb{R^{}}^{})$ we have formally
\begin{equation*}
\left\langle (-\partial^{2}_{x}-\gamma\delta(x))u, v\right\rangle=\mathfrak{t}_{\gamma}(u,v),
\end{equation*}
where $\mathfrak{t}_{\gamma}$ is the bilinear form defined on $H^{1}(\mathbb{R^{}}^{})$ by 
\begin{equation}\label{UKI}
\mathfrak{t}_{\gamma}\left(u,v\right)=\Re\int_{\mathbb{R}}\partial_{x}u\partial_{x}\overline{v} dx-\gamma\Re \left[u(0) \overline{v(0)}\right].
\end{equation}
 It is clear that this form is bounded from below and closed on $H^{1}(\mathbb{R^{}}^{})$.  Then it is possible to show that the self-adjoint operator on $L^{2}(\mathbb{R^{}}^{})$ associated with $\mathfrak{t}_{\gamma}$ is  given by (see \cite[Theorem 10.7 and Example 10.7]{KSN})
 \begin{equation}\label{exten}
\begin{cases}
 H_{\gamma} v(x)=-\frac{d^2}{dx^2} v(x)\qquad \text{for}\;\; x\neq 0,\\
v\in \mathrm{dom}(H_{\gamma})=\left\{v\in H^{1}(\mathbb{R^{}}^{})\cap H^{2}(\mathbb{R }\verb'\'  \left\{0 \right\}) : v'(0+)- v'(0-)=-\gamma v(0)\right\}.
\end{cases}
\end{equation}
%such that for $v\in \mathrm{dom}(H_{\gamma})$ we have
%$$
%H_{\gamma}v(x)=-v''(x),\qquad \text{for}\;\; x\neq 0.
%$$
Notice that $H_{\gamma}$ can also be defined via theory of self-adjoin extensions of symmetric operator (see \cite{APIN, AAN, ASOP}).  Now, from Albeverio {\it et.al} (see \cite[Chapter I.3]{APIN} for details) we have the following spectral properties of  $H_{\gamma}$ which will be used in our local well-posedness theory for model (\ref{00NL}):  for $\sigma_{\rm ess}(H_{\gamma})$ and $\sigma_{\rm disc}(H_{\gamma})$ denoting the essential  and discrete spectrum of $H_{\gamma}$, respectively, it is well known that $\sigma_{\rm ess}(H_{\gamma})=[0,\infty)$, for $\gamma\neq 0$; 
$\sigma_{\rm disc}(H_{\gamma})=\emptyset$ for $\gamma < 0$; $\sigma_{\rm dis}(H_{\gamma})=\left\{-\gamma^{2}/4\right\}$ for $\gamma >0$.

%Therefore, equation (\ref{00NL}) is described exactly by the following boundary problem (see \cite{AFI})
%\begin{equation}\label{SCH02}
%\left \{ 
%\begin{aligned}
 %& i\partial_{t}u(x,t)+ \partial^2_{x}u(x,t)= -u(x,t) \mbox{Log} \left|u(x,t)\right|^{2},\,\,\, x\neq0,\; t \in \mathbb R \\
 %& \lim_{x\to 0^+}[u(x,t)-u(-x, t)]=0, \\
 %& \lim_{x\to 0^+}[\partial_x u(x,t)-\partial_x u(-x, t)]=-\gamma u(0,t), \\
 %& \lim_{x\to \pm \infty} u(x,t)=0,
%\end{aligned} \right.
%\end{equation}
%hence $u(x,t)$ must be solution of the classical  logarithmic Schr\"odinger equation on $\mathbb R^{-}$ and $\mathbb R^{+}$, continuous at $x=0$ and satisfy a ``jump condition'' at the origin and it also vanish at infinity.  
Before presenting our results,  let us first introduce some preliminaries. We consider the reflexive Banach space (see Appendix below)
\begin{equation}\label{W}
W(\mathbb R)=\{u\in H^1(\mathbb R): |u|^2 \mbox{Log}|u|^2 \in L^1(\mathbb R)\}.
\end{equation}
By Lemma \ref{APEX23} in Appendix, we have that the operator 
\begin{equation*}
\begin{cases}
{W}^{}({\mathbb{R}})\rightarrow {W}^{\prime}({\mathbb{R}})\\
u\rightarrow  \partial^{2}_{x}u+ \gamma\delta(x)u+u\,  \mathrm{Log}\left|u\right|^{2}
\end{cases}
\end{equation*}
is continuous and bounded. Here, ${W}^{\prime}({\mathbb{R}})$ is the dual space of ${W}^{}({\mathbb{R}})$. Therefore,  if $u\in C(\mathbb{R}, {W}({\mathbb{R}^{}}))\cap C^{1}(\mathbb{R}, {W}^{\prime}({\mathbb{R}}))$, then equation \eqref{00NL} makes sense in ${W}^{\prime}({\mathbb{R}})$. We define the energy functional
\begin{equation}\label{energy}
E(u)=\frac{1}{2} \|\partial_{x}u \|^{2}_{L^{2}} -\frac{\gamma}{2}\left|u(0)\right|^{2}-\frac{1}{2}\int_{\mathbb{R}}\left|u\right|^{2}\mbox{Log}\left|u\right|^{2}dx.
\end{equation}
At least formally, we have that $E$ is conserved by the flow of (\ref{00NL}) (see Proposition \ref{APCS}). Moreover,  by Proposition \ref{DFFE} in Appendix we  also have that   $E$ is well-defined and of class $C^1$ on $W(\mathbb R)$.

We remark that the use  of the space $W(\mathbb R)$ is mainly due to the fact that the functional $E$, in general, fails to be finite and of class $C^1$ on all $H^1(\mathbb R)$ (see Cazenave \cite{CL}).

The next proposition gives a result on the existence of weak solutions to \eqref{00NL} in the energy space $W(\mathbb R)$. We  recall that a global (weak) solution to \eqref{00NL} is a function $u\in C(\mathbb{R}, W(\mathbb{R}^{}))\cap C^{1}(\mathbb{R}, W^{\prime}(\mathbb{R}^{}) )$ solving \eqref{00NL}  in $W^{\prime}(\mathbb{R}^{})$ for all $t\in \mathbb{R}$.

\begin{proposition} \label{PCS}
For any $u_{0}\in W(\mathbb{R})$, there is a unique maximal solution  $u\in C(\mathbb{R}, W(\mathbb{R}^{}))\cap C^{1}(\mathbb{R}, W^{\prime}(\mathbb{R}^{}) )$  of \eqref{00NL}  such that $u(0)=u_{0}$ and $\sup_{t\in \mathbb{R}}\left\|u(t)\right\|_{W(\mathbb{R})}<\infty$. Furthermore, the conservation of energy and charge hold; that is, 
\begin{equation*}
E(u(t))=E(u_{0})\,\,\,   and \,\,\,\, \left\|u(t)\right\|^{2}_{L^{2}}=\left\|u_{0}\right\|^{2}_{L^{2}},  \text{ for all $t\in \mathbb{R}$}.
\end{equation*}
\end{proposition}

The proof of Proposition \ref{PCS} will be given in Section 2. In this paper,  we are mainly interested in the study of the orbital stability of standing wave solutions $u(x,t)=e^{i\omega t}\varphi(x)$  for  \eqref{00NL}, where $\omega\in \mathbb{R}$, $\gamma>0$, and $\varphi\in W(\mathbb{R}) \cap dom(H_\gamma$)  is a real valued function which has to solve the following stationary  problem
\begin{equation}\label{eq2}
-\partial^{2}_{x} \varphi - \gamma\delta(x)\varphi+\omega \varphi-\varphi\,  \mathrm{Log}\left|\varphi \right|^{2}=0 \quad \text{in} \quad  W^{\prime}(\mathbb{R}).
\end{equation}
As we will see later in Section \ref{S:2}, there exist a unique positive symmetric solution $\phi_{\omega,\gamma}$ (the soliton  peak-Gausson profile) of \eqref{eq2} which is explicitly given for every $\omega\in \mathbb R$ by
\begin{equation}\label{gausson}
\phi_{\omega,\gamma}(x)=e^{\frac{\omega+1}{2}} e^{-\frac{1}{2} (|x|+\frac{\gamma}{2})^2}.
\end{equation}

This  solution is constructed from the known solution of \eqref{00NL} in the case  $\gamma=0$ (namely, $\phi_{\omega, 0}$) on each side of the defect pasted together at $x=0$ to satisfy the continuity and the jump condition  $\phi_{\omega,\gamma}'(0+)-\phi'_{\omega,\gamma}(0-)=-\gamma\phi_{\omega,\gamma}(0)$ at $x=0$. 

The dependence of $\phi_{\omega, \gamma}$ on $\gamma$ can be seen in Fig. \ref{fig:1e2} below. Notice that the sign of $\gamma$ determines the profile of $\phi_{\omega, \gamma}$ near $x=0$. Indeed, it has a ``'$\vee$'  shape when $\gamma<0$, and a ``$\wedge$'' shape when $\gamma>0$.

\begin{figure}[htbp]
\centering
\fbox{\subfigure[]{\includegraphics[width=45mm]{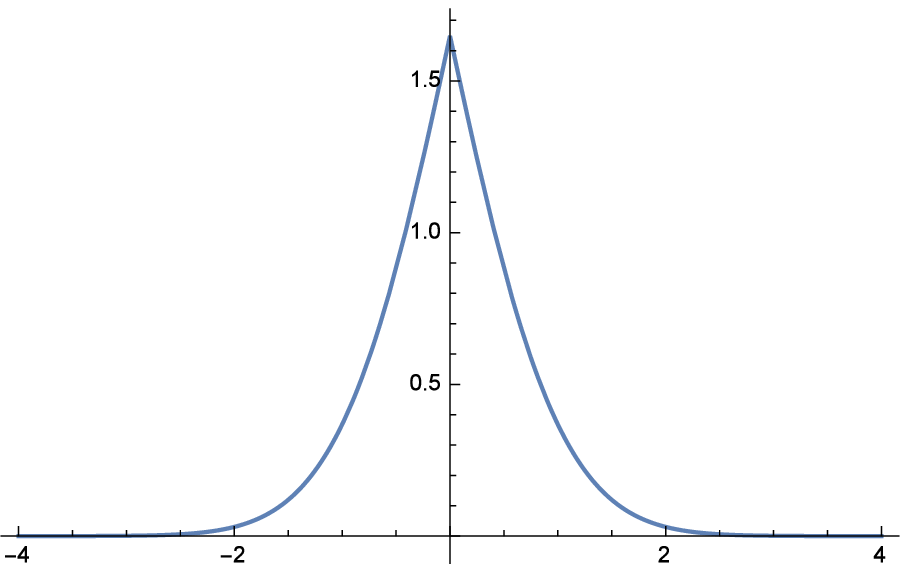}}}\hspace{18mm}
\fbox{\subfigure[]{\includegraphics[width=45mm]{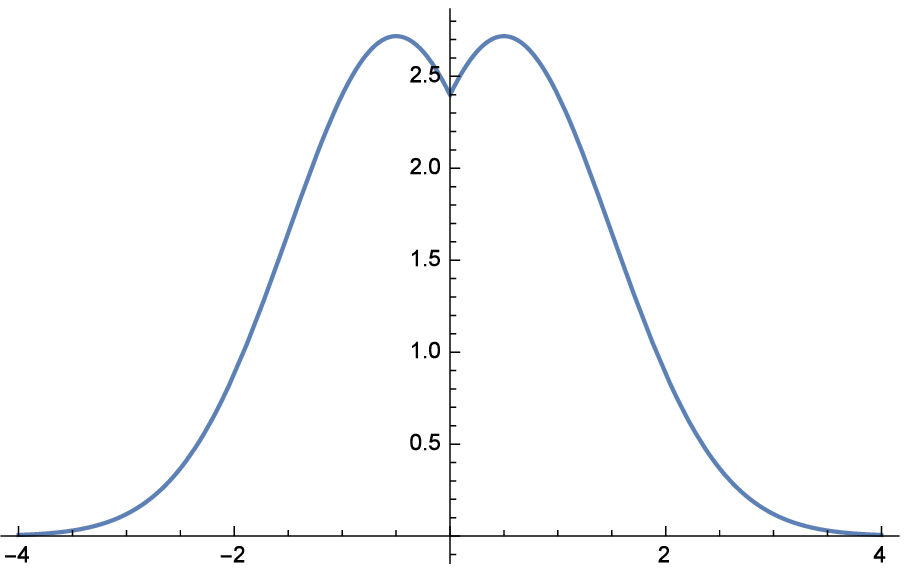}}}
\caption{ \rm{$\phi_{\omega, \gamma}$ as a function of $x$ for  $\omega=1$. (a) $\gamma=1$; (b) $\gamma=-1$.} \label{fig:1e2}}
\end{figure}

For $\gamma=0$, the equation \eqref{00NL} in higher dimensions,
\begin{equation}
\label{00NL2}
 i\partial_{t}u+\Delta u+u\, \mbox{Log}\left|u\right|^{2}=0,
\end{equation}
has been studied previously by several authors (see \cite{BLSE,CAS, CL, CALO, PMS, AH1} and the references therein). In particular,  the Gaussian shape standing-wave for (\ref{00NL2}) (introduced by  Bialynicki-Birula and Mycielski in the '70 \cite{BLSE,CAS})
$$
\varphi_{\omega, N}(x)=e^{\frac{\omega+N}{2}} e^{-\frac{1}{2} |x|^2}
$$ 
in dimension $N$,  it was showed by Cazenave  in  \cite{CAS} that they are orbitally stable in $W(\mathbb{R}^N)=\{u\in H^1(\mathbb R^N): |u|^2 \mbox{Log}|u|^2 \in L^1(\mathbb R^N)\}$ under radial perturbations for $N\geq 2$.  In the case $N\geqq 1$,  we can use the Cazenave-Lions's approach in \cite{CALO} for showing stability on  all $W(\mathbb{R}^N)$. For $N\geqq 3$, d'Avenia\&Montefusco\&Squassina in \cite{PMS} showed the existence of infinitely many weak solutions for 
\begin{equation}\label{logpo}
-\Delta \psi_\omega +\omega\psi_\omega=\psi_\omega \, \mbox{Log} |\psi_\omega|^2,
\end{equation}
and that the Gaussian profile $\varphi_{\omega, N}$ is nondegenerated, that is $Ker(\mathcal L) = span\{\partial_{x_i} \varphi_{\omega, N}: i=1, 2,...N\}$, where $\mathcal Lu = -\Delta u+(|x|^2+ \omega -2)u$ is the linearized operator for $-\Delta u+\omega u = u \, \mbox{Log} (u^2)$.

About model (\ref{00NL}), recently Angulo and Goloshchapova \cite{AAN} proved that $\phi_{\omega,\gamma}$ given by (\ref{gausson}) is orbitally stable  for $\gamma >0$  in the ``weighted space"  \begin{equation*}
\widetilde{W}=H^1(\mathbb R)\cap L^2(x^2\,dx),
\end{equation*}
orbitally unstable in $\widetilde{W}$ for $\gamma <0$, and orbitally stable in $\widetilde{W}_{rad}$ for any $\gamma \neq 0$. The stability analysis in \cite{AAN} relies on the abstract theory by Grillakis, Shatah and Strauss \cite{GST2}, the analytic perturbation theory and extension theory of symmetric operators (see also \cite{AAN2} for applications of the extension theory  in the case of star graphs).
Mention that, since key energetic  functional $E: W(\mathbb{R})\to \mathbb R$ is not twice continuously differentiable at $\phi_{\omega, \gamma}$, the approach elaborated in \cite{GST2} can not be done on $W(\mathbb R)$,  but they can do it on the weighted space  $\widetilde{W}$ which is continuously embedding in $W(\mathbb R)$. This is the  main reason because in \cite{AAN} was necessary to use the space $\widetilde{W}$ in the stability approach.

The purpose in this paper is to extend the results  in  \cite{AAN}  about the  stability of  $\phi_{\omega, \gamma}$ in (\ref{gausson}) to the space $W(\mathbb R)$  for the case $\gamma>0$. Our approach will be based on a variational characterization of $\phi_{\omega, \gamma}$. It characterization can not be used to treat the case $\gamma <0$ and it is left open (see Remark \ref{RM1} below).

The basic symmetry associated to equation \eqref{00NL} is the phase-invariance (while the translation invariance  does not hold due to the defect). Thus,  the definition  of stability  takes into account only  this  type of symmetry and is formulated as follows.

\begin{definition}\label{2D111}
We say that  a standing wave solution $u(x,t)=e^{i\omega t}\phi(x)$ of \eqref{00NL} is orbitally stable in $W(\mathbb{R})$ if for any  $\epsilon>0$ there exist $\eta>0$  such that if $u_{0}\in W(\mathbb{R})$ and $\left\|u_{0}-\phi \right\|_{W(\mathbb{R})}<\eta$, then the solution $u(t)$ of  \eqref{00NL}  with $u(0)=u_{0}$ exist for all $t\in \mathbb R$ and satisfies 
\begin{equation*}
{\rm\sup\limits_{t\in \mathbb R}} {\rm\inf\limits_{\theta\in \mathbb{R}}} \|u(t)-e^{i\theta}\phi \|_{W(\mathbb{R})}<\epsilon.
\end{equation*}
Otherwise, the standing wave $e^{i\omega t}\phi(x)$ is said to be  unstable in $W(\mathbb{R})$.
\end{definition}

Next we state our main result in this paper.

\begin{theorem} \label{2ESSW}
Let $\omega \in \mathbb{R}$. If  $\gamma>0$,  then  the standing wave $e^{i\omega t}\phi_{\omega, \gamma}(x)$, where $\phi_{\omega, \gamma}$ is given in (\ref{gausson}),  is orbitally stable in  $W(\mathbb{R})$.
\end{theorem}

The proof of Theorem \ref{2ESSW} is based on the variational characterization of the stationary solutions $\varphi$ for (\ref{eq2}) as minimizers  of the action $S_{\omega, \gamma} (u)=E(u)+\frac{\omega+1}{2}\|u\|^2_{L^2}$ on the Nehari manifold
$$
\left\{u\in W(\mathbb{R}^{}) \setminus  \left\{0 \right\}:  I_{\omega,\gamma}(u)=0\right\}
$$
with $I_{\omega,\gamma}(u)=2E(u)+\omega\|u\|^2_{L^2}$ (see Theorem \ref{ESSW}), and the uniqueness of positive solutions (modulo rotations) for (\ref{eq2})  given by the peak-Gausson profile (\ref{gausson}) (see Proposition \ref{UDGS}). We remark that an analogous variational analysis have been used for NLS equations with point interactions on all line by Fukuizumi\&Jeanjean \cite{RJJ}, Fukuizumi\&Otha\&Ozawa \cite{FO}, Adami\&Noja \cite{ADNP}, Adami\&Noja\&Visciglia \cite{ADNV}, and on star graphs by  Adami\&Cacciapuoti\newline\&Finco\&Noja \cite{ACFN}.

We also note that recently   has been considered  equation (\ref{00NL2}) with an external potential $V$ satisfying specific conditions (see Ji\&Szulkin \cite{JS} and Squassina\&Szulkin\cite{SS}),
 \begin{equation*}\label{NLSLP}
i\partial_t u+\Delta u -V(x)u+ u\, \mbox{Log}|u|^2=0.
\end{equation*}
 From the results of Ji and Szulkin in \cite{JS} follows that there exist infinitely many profiles of standing wave  $u(x,t)=e^{i\omega t} f_\omega$ (see also  \cite{SS}) for $V$ being coercive. Namely,  the elliptic equation
\begin{equation}\label{lopo}
-\Delta f_\omega +(V(x)+\omega)f_\omega=f_\omega \, \mbox{Log} |f_\omega|^2
\end{equation}
has infinitely many  solutions for $V\in C(\mathbb R^N, \mathbb R)$ such that  $\lim\limits_{|x|\to \infty} V(x)=+\infty$. Moreover they also showed the existence of  a ground state solution (a nontrivial positive solution with least possible energy) for   bounded potential such that  $\omega+1+ V_{\infty}>0$,  in which   
$$
V_{\infty}:=\lim\limits_{|x|\to \infty} V(x)=\sup\limits_{x\in \mathbb R^N} V(x),
$$
 and  $\sigma(-\Delta +V(x)+\omega+1)\subset (0, +\infty)$ (here $\sigma(A)$ represents the spectrum of the linear operator $A$).  Thus, we see that in the case of a delta-potential $V(x)=-\gamma\delta(x)$ the later restriction on the ``frequency'' $\omega$ is ineffective (see proof of Theorem \ref{ESSW} below).

The rest of the paper is organized as follows. In Section \ref{S:0}, we give an idea of the proof of  Proposition \ref{PCS}. In Section  \ref{S:2}, we prove that the stationary problem (\ref{eq2}) has a unique nonnegative nontrivial solution. Section  \ref{S:3} is devoted to give a variational  characterization of the stationary solutions of (\ref{eq2}). In Section \ref{S:4}, we establish  the proof of Theorem \ref{2ESSW}. In the Appendix we include  informations about the space  $W(\mathbb{R})$ and the smooth property of the energy functional $E$ in (\ref{energy}).

\paragraph{\bf Notation:} The space $L^{2}(\mathbb{R},\mathbb{C})$  will be denoted  by $L^{2}(\mathbb{R})$ and its norm by $\|\cdot\|_{L^{2}}$.  This space will be endowed  with the real scalar product
\begin{equation*}
 \left(u,v\right)=\Re\int_{\mathbb{R}}u\overline{v}\, dx, \,\,\,\,\,\,\,\,\,\,\,\,\, \mathrm{for }\,\,\,\,\,\,\,\ u,v\in L^{2}\left(\mathbb{R}\right).
\end{equation*}
The space $H^{1}(\mathbb{R},\mathbb{C})$  will be denoted  by $H^{1}(\mathbb{R})$ and its norm by $\|\cdot\|_{H^{1}(\mathbb{R})}$.
The dual space of  $H^{1}(\mathbb{R})$ will be denoted by $H^{-1}(\mathbb{R})$. We denote by $C_{0}^{\infty}\left(\mathbb{R}\setminus\left\{0\right\}\right)$ the set of $C^{\infty}$ functions from $\mathbb{R}\setminus\left\{0\right\}$ to $\mathbb{C}$ with compact support. Throughout this paper, the letter $C$ will denote positive constants.\\

\section{The Cauchy problem}
\label{S:0}

In this section we prove the well-posedness of the Cauchy Problem for \eqref{00NL} in the energy space $W(\mathbb{R})$. The idea of the proof is an adaptation of the proof of \cite[Theorem 9.3.4]{CB}.  So, we will approximate the logarithmic nonlinearity by a smooth nonlinearity, and as a consequence we construct a sequence of global solutions of the regularized Cauchy problem in $C(\mathbb{R}, H^1(\mathbb{R}))$,  then we pass to the limit using standard compactness results, extract a subsequence which converges to the solution of the limiting equation \eqref{00NL}.

First, let  us  establish  the  following  well-posedness  result  in $ H^{1}(\mathbb{R})$ associated  with the NLS equation with a delta potential
\begin{equation}\label{HAX}
\begin{cases}
i\partial_{t}u+\partial^{2}_{x}u+ \gamma\delta(x)u+g(u)=0,\\
u(0)=u_{0}\in H^{1}(\mathbb{R}),
\end{cases}
\end{equation}
where $g: L^{2}(\mathbb{R})\rightarrow L^{2}(\mathbb{R})$ is  globally Lipschitz continuous  on $ L^{2}(\mathbb{R})$, $\Im(g(u),iu)=0$, and such that there exist $G\in C(H^1(\mathbb{R}),\mathbb{R})$ with $G^{\prime}=g$.
\begin{proposition} \label{APCS}
For any $u_{0}\in H^{1}(\mathbb{R})$, there is a unique maximal solution  $u\in C(\mathbb{R},H^{1}(\mathbb{R}))\cap C^{1}(\mathbb{R}, H^{-1}(\mathbb{R}))$  of \eqref{HAX} such that $u(0)=u_{0}$. Furthermore, the conservation of charge and energy hold; that is, for all $t\in \mathbb{R}$, $\left\|u(t)\right\|^{2}_{L^{2}}=\left\|u_{0}\right\|^{2}_{L^{2}}$ and
\begin{equation*}
\mathcal{E}(u(t))=\mathcal{E}(u_{0}), \quad \text{where} \quad \mathcal{E}(u)=\frac{1}{2}\left\|\partial_{x} u\right\|^{2}_{L^{2}}-\frac{\gamma}{2}\left|u^{}(0)\right|^{2} -G(u).
\end{equation*}
\end{proposition}
\begin{proof} The idea will be to check the assumptions of  Theorem 3.3.1 and Theorem 3.7.1 in  \cite{CB} for obtaining the local result.  Indeed, first we note that $H_\gamma$ defined in \eqref{exten} satisfies $H_\gamma\geq -m$, where $m = \gamma^2/4$ if $\gamma > 0$, and $m = 0$ if $\gamma < 0$. Thus, we have that $A \equiv - H_\gamma-m $ is a self-adjoint operator, $A\leqq 0$ on $X = L^2(\mathbb{R})$ with domain $\mathrm{dom}(A) = \mathrm{dom}(H_\gamma)$. Moreover, in our case the  following norm on $H^1(\mathbb{R})$
\begin{equation*}
\left\|u\right\|^2_{X_A}=\left\|\partial_{x} u\right\|_{L^2}^2+(m+1)\left\|u\right\|_{L^2}^2-\gamma|u(0)|^2,
\end{equation*}
is  equivalent to the usual $H^1(\mathbb{R})$-norm. Next, it is easy see that the conditions (3.7.1), (3.7.3)-(3.7.6) in \cite[Section 3.7]{CB} hold choosing $r=\rho=2$, because we are in one dimensional case.  Also, the condition (3.7.2)  follows easily from the self-adjoint property of $A$. Lastly, we need to show that there is uniqueness for the problem \eqref{HAX}. Thus, let $I$ be an interval containing $0$ and let $u_{1}$, $u_{2}\in L^{\infty}(I, H^1(\mathbb{R}))\cap W^{1,\infty}(I, H^1(\mathbb{R}))$ be two solutions of \eqref{HAX}. It follows that (see \cite[Remark 3.7.2]{CB})
\begin{equation*}
u_{2}(t)-u_{1}(t)=i\,\int^{t}_{0}e^{-i\,A(t-s)}\left(g(u_{2}(s))-g(u_{1}(s))\right)ds \quad \text{for all} \quad t\in I.
\end{equation*}
Since $g$ is  globally Lipschitz continuous on $L^2(\mathbb R)$, there exist a constant  $C$ such that 
\begin{equation*}
\|u_{2}(t)-u_{1}(t)\|^{2}_{L^{2}}\leq C \int^{t}_{0}\|u_{2}(s)-u_{1}(s)\|^{2}_{L^{2}}ds ,
\end{equation*}
and therefore the uniqueness  follows by Gronwall's Lemma. Therefore, we obtain that the  initial value problem (\ref{HAX}) is locally well posed in $(H^1(\mathbb{R}), \|\cdot\|_{X_A})$. Moreover, we have the conservation of charge and energy. 

Finally, from Theorem 3.3.1 and Theorem 3.7.1 in  \cite{CB} we see easily that the solution of  (\ref{HAX}) is global in $H^{1}(\mathbb{R})$. This finishes the Proposition.
\end{proof}

\begin{remark}
 For the  completeness of the exposition we recall that for $\gamma<0$ the unitary group $G_\gamma(t)=e^{-it H_\gamma}$ associated to  equation \eqref{HAX} is given explicitly  by the formula (see \cite{Holmer5}),
 \begin{equation*}\label{pospro1}
G_\gamma(t) \phi(x)= e^{it\partial_x^2} \phi (x)  + e^{it \partial_x^2} (\phi\ast \rho_\gamma) (|x|),
\end{equation*}
where
$$
\rho_\gamma(x)=-\frac{\gamma}{2} e^{\frac{\gamma}{2} x}\chi^0_{+}.
$$
Here  $\chi^0_{+}$   denotes  the characteristic function of $[0,+\infty)$ and $e^{it\partial_x^2}$ represents the free group of Schr\"odinger ($\gamma =0$). For the case $\gamma>0$ we refer to  \cite{DH}. Thus, an explicit formula for the group $e^{it A}$ is possible to be  obtained.
\end{remark}

\begin{proof}[ \bf {Proof of Proposition \ref{PCS}}]
The proof is an adaptation  of the proof of \cite[Theorem 9.3.4]{CB}. We only  discuss the  modifications that are not sufficiently clear in our case. We first regularize the logarithmic nonlinearity near the origin. Indeed, for $z\in \mathbb{C}$,  we define the functions $a_{m}$ and $b_{m}$ by 
\begin{equation*}
a_{m}(z)=
\begin{cases}
 a_{}(z), &\text{if $\left|z\right|\geq \frac{1}{m}$;}\\
m\,z\,a_{}(\frac{1}{m}) , &\text{if $\left|z\right|\leq \frac{1}{m}$;}
\end{cases}
\,\,\,\,\,\,\,  \text{and} \,\,\,\,\,\,
b_{m}(z)=
\begin{cases}
 b_{}(z) , &\text{if $\left|z\right|\leq {m}$;}\\
\frac{z}{m}\,b({m}) , &\text{if $\left|z\right|\geq {m}$,}
\end{cases}
\end{equation*}
where  the functions $a$ and  $b$ are defined in \eqref{abapex} in Appendix. Moreover,  set $f_{m}(u)=b_{m}(u)-a_{m}(u)$  for  $u\in H^{1}(\mathbb{R})$. We remark that  the function $f_{m}$ is globally Lipschitz $L^{2}(\mathbb{R})\rightarrow L^{2}(\mathbb{R})$. For a given initial data $u_{0}\in H^{1}(\mathbb{R})$, we consider the following  regularized Cauchy problem
\begin{equation}
\label{3SAP}
 \left\{
\begin{array}{ll}
i\partial_{t}u^{m}+ (\partial^{2}_{x} + \gamma\delta(x))u^{m}+f_{m}(u^{m})=0,\\
u^{m}(0)=u_{0}.
\end{array}\right.
\end{equation}
%Next, we remark that $ H_\gamma$ defined in \eqref{exten} satisfies $ H_\gamma\geq -m$, where $m = \gamma^2/4$ if $\gamma > 0$, and $m = 0$ if $\gamma < 0$. Thus, $A =- H_\gamma-m $ is a self-adjoint negative operator on $X = L^2(\mathbb{R})$ on the  domain $\mathrm{dom}(A) = \mathrm{dom}(H_\gamma)$. Moreover, in our case the  norm
%$$
%||v||^2_{X_A}=||v'||_{L^2}^2+(m+1)||v||_{L^2}^2-\gamma|v(0)|^2,
%$$
%is  equivalent to the usual $H^1(\mathbb{R})$-norm. Therefore, from  \cite[Theorem 3.7.1 and Corollary 3.3.11]{CB}, 
Applying Proposition \ref{APCS}, we  see that  for every $m\in \mathbb{N}$  there exist a unique global (weak) solution $u^{m}\in C(\mathbb{R}, H^{1}(\mathbb{R}))\cap C^{1}(\mathbb{R},H^{-1}(\mathbb{R}) )$ of \eqref{3SAP} which satisfies 
\begin{equation*}
E_{m}(u^{m}(t))=E_{m}(u_{0}) \,\, \mbox{and} \,\, \left\|u^{m}(t)\right\|^{2}_{L^{2}}=\left\|u_{0}\right\|^{2}_{L^{2}},  \,\, \text{ for all $t\in \mathbb{R}$},
\end{equation*}
where \begin{equation*}
E_{m}(u)=\frac{1}{2}\left\|\partial_{x} u^{}(t)\right\|^{2}_{L^{2}}-\frac{\gamma}{2}\left|u^{}(0,t)\right|^{2}  +\frac{1}{2}\int_{\mathbb{R}}\Phi_{m}(\left|u_{}\right|)dx-\frac{1}{2}\int_{\mathbb{R}}\Psi_{m}(\left|u_{}\right|)dx,
\end{equation*}
and the functions $\Phi_{m}$ and $\Psi_{m}$  are defined by
\begin{equation*}
\Phi_{m}(z)=\frac{1}{2}\int^{\left|z\right|}_{0}a_{m}(s)ds\,\,\,\ \mbox{and} \,\,\,\ \Psi_{m}(z)=\frac{1}{2}\int^{\left|z\right|}_{0}b_{m}(s)ds.
\end{equation*}
Arguing in the same way as in the proof of  Step 2 of \cite[Theorem 9.3.4]{CB} we deduce that the sequence of approximating solutions $u^{m}$ is bounded in the space $L^{\infty}(\mathbb{R}, H^{1}(\mathbb{R}))$. It also follows from the NLS equation \eqref{3SAP} that the sequence $\left\{\left.u^{m}\right|_{\Omega_{k}}\right\}_{m\in\mathbb{N}}$ is bounded in the space $W^{1,\infty}(\mathbb{R}, H^{-1}(\Omega_{k}))$, where $\Omega_{k}=(0,k)$. Therefore, we have that  $\left\{u^{{m}}\right\}_{m\in\mathbb{N}}$ satisfies the assumptions of  Lemma 9.3.6 in \cite{CB}. Let $u$ be the limit of $u^{m}$. 

We show that the limiting function $u\in L^{\infty}(\mathbb{R},H^{1}(\mathbb{R}))$ is a weak solution of \eqref{00NL}.  We first write  a weak formulation of the NLS equation \eqref{3SAP}. Indeed, for any test  functions $\psi\in C^{\infty}_{0}({\mathbb{R}_x})$  and $\phi\in C^{\infty}_{0}({\mathbb{R}_t})$, we have
\begin{equation}\label{3DPL}
-\int_{\mathbb{R}}\left[\left\langle i\, u^{m}, \psi\right\rangle \phi^{\prime}(t)+\mathfrak{t_{\gamma}}(u^{m},\psi) \phi^{}(t)\right]\,dt+\int_{\mathbb{R}}\int_{\mathbb{R}}f_{m}(u^{m})\psi(x)\phi(t)\,dx\,dt=0.
\end{equation}

We pass to the limit as $n\rightarrow\infty$ in  the integral formulation \eqref{3DPL} and obtain the following integral equation (see proof of Step 3 of \cite[Theorem 9.3.4]{CB}),
\begin{equation}\label{ert}
-\int_{\mathbb{R}}\left[\left\langle i\, u, \psi\right\rangle \phi^{\prime}(t)+\mathfrak{t_{\gamma}}(u,\psi) \phi^{}(t)\right]\,dt+\int_{\mathbb{R}}\int_{\mathbb{R}}u\, \mbox{Log}\left|u\right|^{2}\psi(x)\phi(t)\,dx\,dt=0.
\end{equation}
Moreover,  it is easy to see that $u\in{L^{\infty}(\mathbb{R}, L^{A}(\mathbb{R}))}$ and  $u(0)=u_{0}$. Therefore, by integral equation \eqref{ert},  $u\in{L^{\infty}(\mathbb{R}, W(\mathbb{R}))}$ is a weak solution of the equation \eqref{00NL}. In particular, from Lemma \ref{APEX23} in Appendix, we deduce that $u\in W^{1,\infty}(\mathbb{R}, W^{\prime}(\mathbb{R}))$. 

Now we show uniqueness the solution in the class $L^{\infty}(\mathbb{R}, W^{}(\mathbb{R}))\cap W^{1,\infty}(\mathbb{R}, W^{\prime}(\mathbb{R}))$. Indeed, let $u$ and $v$ be two solutions of \eqref{00NL} in that class. On taking the difference of the two equations and taking the $W(\mathbb{R})-W^{\prime}(\mathbb{R})$ duality product with $i(u-u)$, we see that
\begin{equation*}
\left\langle u_{t}-v_{t}, u-v\right\rangle_{W(\mathbb{R}),W^{\prime}(\mathbb{R})}=-\Im \int_{\mathbb{R}}\left(u\mathrm{Log}\left|u \right|^{2}-v\mathrm{Log}\left|v \right|^{2} \right)(\overline{u}-\overline{v})dx.
\end{equation*}
Thus, from \cite[Lemma 9.3.5]{CB} we obtain
\begin{equation*}
\left\|u(t)-v(t)\right\|_{L^{2}}\leq 8\int^{t}_{0}\left\|u(s)-v(s)\right\|^{2}_{L^{2}}ds.
\end{equation*}
Therefore, the uniqueness of a solution follows by Gronwall's Lemma.  

We claim that the weak solution $u$ of  \eqref{00NL} satisfies both conservation of charge and energy. Indeed, by weak lower semicontinuity of the $H^{1}(\mathbb{R})$-norm, Fatou's lemma and arguing in the same way as in the proof of  Step 3 of \cite[Theorem 9.3.4]{CB} we deduce that
\begin{equation}\label{LON}
E(u(t))\leq E(u_{0}) \quad \mbox{and} \quad \left\|u^{}(t)\right\|^{2}_{L^{2}(\mathbb{R})}=\left\|u_{0}\right\|^{2}_{L^{2}} \quad \text{for all} \,\, t\in \mathbb{R}.
\end{equation}
Now fix $t_{0}\in \mathbb{R}$. Let $\varphi_0=u(t_{0})$ and let $w$ be the solution of \eqref{00NL} with $w(0)=\varphi_0$. By uniqueness, we see that $w(\cdot-t_{0})=u(\cdot)$ on $\mathbb{R}$. From \eqref{LON}, we deduce in particular that  
\begin{equation*}
E(u_{0})\leq E(\varphi_0)= E(u(t_{0})).
\end{equation*}
Therefore, we have that both $\left\|u^{}(t)\right\|^{2}_{L^{2}}$ and $E(u(t))$ are constant on $\mathbb{R}$.
Finally, the inclusion $u\in C(\mathbb{R}, W(\mathbb{R}))\cap C^{1}(\mathbb{R}, W^{\prime}(\mathbb{R}))$ follows from conservation laws. This completes the proof.
\end{proof}

\section{Stationary problem}
\label{S:2}

This section is devoted to show that the following set,  
\begin{equation*}
\mathcal A_{\omega, \gamma}=\{\varphi\in W(\mathbb R)\setminus\{0\}: \varphi\; \text{is solution of the stationary problem (\ref{eq2}) in}\; W'(\mathbb R)\}
\end{equation*}
it is given (modulo rotations) by  $\phi_{\omega, \gamma}$ in (\ref{gausson}). More exactly we have the following result.

\begin{proposition} \label{UDGS}
Let $\gamma\in\mathbb{R}\setminus \left\{0 \right\}$ and $\omega\in \mathbb{R}$. Then  \eqref{eq2} has a unique nonnegative nontrivial solution. Thus,  $\phi_{\omega, \gamma}$ is this solution and therefore $\mathcal{A}_{\omega,\gamma}=\left\{e^{i\theta}\phi_{\omega, \gamma}: \theta\in\mathbb{R} \right\}$.
\end{proposition}

For $\gamma=0$, the set of solutions  of the stationary problem \eqref{eq2} is well known. In particular, modulo translation and phase, there exist a unique positive solution, which is explicitly known. Indeed, $\mathcal{A}_{\omega,0}=\left\{e^{i\theta}\phi_{\omega,0}(\cdot-y); \theta\in\mathbb{R}, y\in\mathbb{R}\right\}$ (see, for example, \cite[Appendix D]{CAS}). \\

Before to give the proof of Proposition \ref{UDGS}, we have the following  basic properties of the solutions of \eqref{eq2}.

\begin{lemma} \label{LCU}
 Let $\gamma\in\mathbb{R}\setminus \left\{0\right\}$,  $\omega\in \mathbb{R}$ and $v\in\mathcal{A}_{\omega,\gamma}$. Then, $v$ verifies the following:
\begin{align}
& v\in C^{j}(\mathbb{R}\setminus \left\{0 \right\})\cap C^{}(\mathbb{R}), \,\,\,\,\, j=1, 2, \label{1} \\
& -v''+\omega v-v\,  \mathrm{Log}\left|v \right|^{2}=0,\,\,\,  \mbox{on}\,\,\, \mathbb{R}\setminus \left\{0 \right\},\label{2}  \\
& v'(0+) -v'(0-) =- \gamma v(0), \label{3}\\
&  v'(x), v(x)\rightarrow 0,  \,\,\,  \mbox{as}\,\,\, \left|x\right|\rightarrow\infty. \label{4}
\end{align}
\end{lemma}
\begin{proof}
The proof of item \eqref{1} follows
 by a standard bootstrap argument using test functions $\xi \in C^{\infty}_{0}(\mathbb{R} \setminus \left\{0 \right\})$ (see, for example, \cite[Chapter 8]{CB}). Indeed, from \eqref{eq2} applied with $\xi\,v $  we deduce that 
\begin{equation*}
-(\xi v)''+\omega \xi v=-\xi''v-2 \xi' v'+\xi v\, \mbox{Log}(|v|^{2}), 
\end{equation*}
in the sense of distributions on $\mathbb{R}\setminus \left\{0 \right\}$. The right hand side is in ${L^{2}(\mathbb{R})}$  and so $\xi v\in H^{2}(\mathbb{R})$. This implies that $v$ is in $C^{2}(\mathbb{R}\setminus \left\{0 \right\})$ and it is a classical solution of this equation on $\mathbb{R}\setminus \left\{0 \right\}$, from which \eqref{1} and \eqref{2} follow. To prove \eqref{3}, we consider $\xi\in W(\mathbb R)$ such that $\xi \in C^{\infty}_{0}(\mathbb{R})$ and for $\epsilon>0$ we have $supp\; \xi\subset [-\epsilon, \epsilon]$ and $\xi(0)=1$. Therefore, by considering $v$ and $\xi$ real valued-functions without loss of generality, we have  for $0<\delta<\epsilon$ that
\begin{equation}
\begin{aligned}
&0=\langle -v''-\gamma \delta(x) v+\omega v-v\mathrm{Log}|v|^2, \xi\rangle\\
&=\lim_{\delta\to 0}\int_{-\epsilon}^{-\delta} v'(x)\xi'
(x)dx+\lim_{\delta\to 0}\int_{\delta}^{\epsilon}v'(x)\xi'
(x)dx-\gamma v(0)\xi(0)\\
& + \omega \int_{-\epsilon}^\epsilon v(x) \xi(x) dx- \int_{-\epsilon}^\epsilon v(x) \mathrm{Log} |v(x)|^2\xi(x) dx\\
&=v'(0-)-v'(0+)-\int_{-\epsilon}^{0-} v''\xi dx-\int_{0+}^\epsilon v''\xi dx-\gamma v(0)+ \omega \int_{-\epsilon}^\epsilon v(x) \xi(x) dx\\
&- \int_{-\epsilon}^\epsilon v \mathrm{Log}|v|^2\,\xi dx \to v'(0-)-v'(0+)-\gamma v(0),
\end{aligned}
\end{equation}
as $\epsilon\to 0$. This proof the jump condition. Finally, since $v\in H^{1}(\mathbb{R}^{})$, it follows that $v(x)\rightarrow 0$ as $\left|x\right|\rightarrow \infty$. Thus, by \eqref{2}, $v^{\prime\prime}(x)\rightarrow 0$ as $\left|x\right|\rightarrow \infty$, and so $v^{\prime}(x)\rightarrow 0$ as $\left|x\right|\rightarrow \infty$. This completes the proof of Lemma \ref{LCU}.
\end{proof}

\begin{remark} An example of test function $\xi\in W(\mathbb R)$ used in the proof of Lemma  \ref{LCU} with the properties that for $\epsilon>0$ we have $supp\; \xi\subset [-\epsilon, \epsilon]$ and $\xi(0)=1$, it is the following
\begin{equation*}
\xi(x)=
\begin{cases}
  e\,e^{\frac{\epsilon^2}{x^2-\epsilon^2}},\;\;\,\;\; \text{if}\;\; |x|<\epsilon\\
0,\quad\;\;\;\;\;\;\;\;\;\;\text{if}\;\; |x|\geqq \epsilon.
\end{cases}
\end{equation*}

\end{remark}

Now we give the proof of  Proposition \ref{UDGS}. Our proof is inspired by the techniques of \cite[Lemma 26]{RJJ}.

\begin{proof}[ {\bf{Proof of Proposition \ref{UDGS}}}] By construction, we have that  $\phi_{\omega, \gamma}\in\mathcal{A}_{\omega,\gamma}$, and so 
$$
\left\{e^{i\theta}\phi_{\omega, \gamma}: \theta\in\mathbb{R} \right\}\subseteq \mathcal{A}_{\omega,\gamma}.
$$ 
Now, let $\varphi\in\mathcal{A}_{\omega,\gamma}$.  Arguing as in  \cite[Lemma 3]{KEO} we can show that there exist $\theta_{0}\in\mathbb{R}$ and a positive function $v$  such that $\varphi(x)=e^{i\theta_{0}}v(x)$ for all $x\in\mathbb{R}$. We will prove that $v(x)=\phi_{\omega, \gamma}(x)$. Indeed, it is clear, by Lemma \ref{LCU}, that the properties  \eqref{1}-\eqref{4} also hold for $v(x)$. Let  $f(s)=-\omega s+s\, \mbox{Log}(s^{2}) $ and $F(s)=\int^{s}_{0}f(t)dt$. Multiplying the equation \eqref{2} by $v'(x)$ and integrating from $x=0$ and $R>0$ yields
\begin{equation*}
-\frac{1}{2}\left\{v'(R)\right\}^{2}+\frac{1}{2}\left\{v'(0+)\right\}^{2}-F(v(R))+F(v(0+))=0.
\end{equation*}
Now, letting $R\rightarrow\infty$, we have
\begin{equation}
\label{C1}
\frac{1}{2}\left\{v'(0+)\right\}^{2}+F(v(0+))=0.
\end{equation}
Arguing in the same way on $\left(-\infty,\, 0\right]$, we conclude that
\begin{equation}
\label{C145}
\frac{1}{2}\left\{v'(0-)\right\}^{2}+F(v(0-))=0.
\end{equation}
Since $v\in W(\mathbb{R})$, $v$ is continuous at 0, thus we must have  $\left|v'(0-)\right|=\left|v'
(0+)\right|$.
Now, if we suppose that $v'(0-)=v'(0+)$ then $v(0)=0$ by \eqref{3}. Then for the case $v'(0-)=v'(0+)=0$  we obtain  immediately $v\equiv 0$ on $\mathbb{R}$.  On the other hand, for $v'(0-)= v'(0+)\neq0$ then $v$ becomes negative  close to $0$. Therefore, since $v$ is a positive solution, this is a contradiction. Hence, we need to have  $v'(0-)=-v'(0+)$, from which we infer immediately that 
\begin{equation}
\label{C2}
v'(0+)=-\frac{\gamma}{2}v(0).
\end{equation}
For $c>0$, we set
\begin{equation*}
P(c)=\frac{\gamma^{2}-4\omega-4}{8}c^{2} +c^{2}\mbox{Log}(c^{}).
\end{equation*}
It is clear that this function has a unique zero $c_{0}>0$, $c_{0}=e^{\frac{\omega+1}{2}}e^{-\frac{\gamma^{2}}{8}}$. Direct computations show that, by \eqref{C1} and \eqref{C2}, the function $v$ in $x=0$ satisfies  $P(v(0))=0 $. Therefore,  
\begin{equation}
\label{C3}
v(0)=c_{0}=e^{\frac{\omega+1}{2}}e^{-\frac{\gamma^{2}}{8}}.
\end{equation}
We note that the initial value problem on $\left(0,\, \infty\right)$ for the equation \eqref{2} with \eqref{C3} and \eqref{C2} as initial conditions has a unique solution. Indeed, the solution is unique for $x>0$ close to $0$ since $f\in C\left(\left[0,\,\infty\right)\right)\cap C^{1}\left(\left(0,\,\infty\right)\right)$. A similar argument can be applied on $\left(-\infty,\, 0\right)$, thus the solution of \eqref{2} is unique in $\left(-\infty,\, 0\right)$. Moreover, we remark that $v(0)=\phi_{\omega, \gamma}(0)$. By the uniqueness, we see that  $v(x)=\phi_{\omega, \gamma}(x)$ on $\mathbb{R}$. This proves $\mathcal{A}_{\omega,\gamma}\subseteq\left\{e^{i\theta}\phi_{\omega, \gamma}: \theta\in\mathbb{R} \right\}$.
\end{proof}

\section{Existence of a ground state}
\label{S:3}

The idea of this section is to give a variational characterization of the stationary solutions for (\ref{eq2}) . It characterization will be used in the stability theory in $W(\mathbb R)$ of the orbit generated by $\phi_{\omega, \gamma}$, $\mathcal A_{\omega, \gamma}$. In order to establish our main result  (Theorem 1.3) we need to establish some preliminaries.

\begin{definition}
For $\gamma\in\mathbb{R}$ and $\omega\in \mathbb{R}$, we define the following functionals of class $C^{1}$ on $W(\mathbb{R})$:
\begin{align*}
 S_{\omega,\gamma}(u)&=\frac{1}{2} \|\partial_{x}u \|^{2}_{L^{2}}+ \frac{\omega+1}{2}\|u \|^{2}_{L^{2}} -\frac{\gamma}{2}\left|u(0)\right|^{2}  -\frac{1}{2}\int_{\mathbb{R}}\left|u\right|^{2}\mbox{Log}\left|u\right|^{2}dx,\\
 I_{\omega,\gamma}(u)&= \|\partial_{x}u \|^{2}_{L^{2}}+ \omega\,\|u \|^{2}_{L^{2}}-\gamma\left|u(0)\right|^{2}-\int_{\mathbb{R}}\left|u\right|^{2}\mbox{Log}\left|u\right|^{2}dx.
\end{align*}
\end{definition}
We note that for $u\in W(\mathbb{R})$ the  derivative of $S_{\omega,\gamma}$ in $u$ is given by
\begin{equation*}
S_{\omega,\gamma}^{\prime}(u)= -\partial^{2}_{x} u -\gamma\delta(x)u +\omega u-u\, \mbox{Log}\left|u\right|^{2},
\end{equation*}
in the sense that for $h\in W(\mathbb{R})$,
$$
S_{\omega,\gamma}^{\prime}(u)(h)= \langle S_{\omega,\gamma}^{\prime}(u), h\rangle= \Re\Big [ \int_{\mathbb R} u'\overline{h'}dx - \int_{\mathbb R} u\overline{h}\mbox{Log} |u|^2dx +\omega  \int_{\mathbb R} u\overline{h}dx\Big ]-\gamma \Re (u(0)\overline{h(0)}).
$$
Therefore, from Lemma \ref{LCU} we have immediately that $\varphi \in \mathcal A_{\omega, \gamma}$  if and only if  $\varphi\in W(\mathbb{R})\setminus\left\{0\right\}$  and $S_{\omega,\gamma}^{\prime}(\varphi)=0$. Indeed, since for $\varphi \in \mathcal A_{\omega, \gamma}$ we have for every $h\in W(\mathbb{R})$
$$
 \Re\int_{\mathbb R} \varphi'\overline{h'}dx= \Re[(\varphi'(0-)-\varphi'(0+))\overline{h(0)}]-\int_{-\infty}^{0-}  \varphi''\overline{h}dx- \int_{0+}^{+\infty}  \varphi''\overline{h}dx,
$$
we obtain immediately from  (\ref{2})-(\ref{3}) that $S_{\omega,\gamma}^{\prime}(\varphi)(h)=0$. The other implication is trivial.
\\

Next, we consider the minimization problem
\begin{align}
\begin{split}\label{MPE}
d_{\gamma}(\omega)&={\inf}\left\{S_{\omega,\gamma}(u):\, u\in W(\mathbb{R}^{})  \setminus  \left\{0 \right\},  I_{\omega,\gamma}(u)=0\right\} \\ 
&=\frac{1}{2}\,{\inf}\left\{\left\|u\right\|_{L^{2}}^{2}:u\in  W(\mathbb{R}^{}) \setminus \left\{0 \right\},  I_{\omega,\gamma}(u)= 0 \right\}, 
\end{split}
\end{align}
and define the set of ground states  by
\begin{equation*}
 \mathcal{N}_{\omega,\gamma}=\left\{ \varphi\in W(\mathbb{R}^{}): S_{\omega,\gamma}(\varphi)=d_{\gamma}(\omega), \,\, I_{\omega,\gamma}(\varphi)=0\right\}.
\end{equation*}
\begin{remark}
The set $\left\{u\in W(\mathbb{R}^{}) \setminus  \left\{0 \right\}:  I_{\omega,\gamma}(u)=0\right\}$ is called the Nehari manifold. By definition,  we have $I_{\omega,\gamma}(u)=\left\langle S_{\omega,\gamma}^{\prime}(u),u\right\rangle$. Thus, the above set is a one-codimension manifold that contains all stationary point of $S_{\omega,\gamma}$.
\end{remark}
\begin{remark}\label{RM}
We have the relation $\mathcal{N}_{\omega,\gamma}\subseteq \mathcal{A}_{\omega,\gamma}$. Indeed, let $u\in \mathcal{N}_{\omega,\gamma}$.  Then, there is a Lagrange multiplier  $\Lambda\in \mathbb{R}$ such that $S^{\prime}_{\omega,\gamma}( u)=\Lambda I^{\prime}_{\omega,\gamma}(u)$. Thus, we have $\left\langle S^{\prime}_{\omega,\gamma}( u),u\right\rangle=\Lambda\left\langle  I^{\prime}_{\omega,\gamma}(u),u\right\rangle$. The fact that   $\left\langle S^{\prime}_{\omega,\gamma}( u),u\right\rangle =I_{\omega,\gamma}(u)=0$ and $\left\langle  I^{\prime}_{\omega,\gamma}(u),u\right\rangle= -2\left\|u_{}\right\|_{L^{2}(\mathbb{R})}^{2}<0$, implies $\Lambda=0$; that is, $S_{\omega,\gamma}^{\prime}(u)=0$ and so $u\in \mathcal{A}_{\omega,\gamma}$. 
\end{remark}
For $\gamma>0$, the existence of minimizers for \eqref{MPE}  is obtained through variational techniques (see \cite{ADNV}, \cite{RJJ}, \cite{FO}) . More precisely, we will show the following theorem.
\begin{theorem} \label{ESSW}
Let  $\gamma>0$ . There exists a minimizer of $d_{\gamma}(\omega)$ for any $\omega \in \mathbb{R}$. In addition, the infimum is achieved at solutions to \eqref{eq2}. More precisely, the set of minimizers for the problem \eqref{MPE} is  given by $\mathcal{N}_{\omega,\gamma}=\left\{e^{i\theta}\phi_{\omega, \gamma}: \theta\in\mathbb{R} \right\}$.
\end{theorem}
\begin{remark}\label{RM1}
We note that  $d_{\gamma}(\omega)$ has no minimizer when $\gamma<0$. We can see  this by contradiction. Suppose that $u_{\omega,\gamma}$  is a minimizer of $d_{\gamma}(\omega)$. From  Remark \ref{RM}, it is clear that there exist $\theta_{0}\in\mathbb{R}$ such that $u_{\omega,\gamma}(x)=e^{i\theta_{0}}\phi_{\omega_{}, \gamma}(x)$. In particular, $\lim_{\left|x\right| \to \infty}\left|u_{\omega,\gamma}(x)\right|= 0$ and  $\left|u_{\omega,\gamma}(x)\right|>0$ on all $x\in\mathbb{R}$.  Now, let $\tau_{y}u_{\omega,\gamma}(x)\equiv u_{\omega,\gamma}(x-y)$ for any $y\in\mathbb{R}$. By direct computations, we see that 
$$I_{\omega,\gamma}(\tau_{y}u_{\omega,\gamma})-I_{\omega,\gamma}(u_{\omega_{}, \gamma})=\gamma [\left|\phi_{\omega_{}, \gamma}(0)\right|^{2}-\left|\phi_{\omega_{}, \gamma}(y)\right|^{2}],$$
 and therefore we have that $I_{\omega,\gamma}(\tau_{y}u_{\omega,\gamma})<0$ for $\left|y\right|$ sufficiently large.
%(see Figure \ref{caja} below).
Thus, there is $\lambda \in (0, 1)$ such that $I_{\omega,\gamma}(\lambda\tau_{y}u_{\omega,\gamma})=0$. Then, by \eqref{MPE} we have 
\begin{equation*}
d_{\gamma}(\omega)\leq\frac{1}{2}\left\|\lambda\,\tau_{y}u_{\omega,\gamma}\right\|_{L^{2}}^{2} <\frac{1}{2}\left\|\tau_{y}u_{\omega,\gamma}\right\|_{L^{2}}^{2}=\frac{1}{2}\left\|u_{\omega_{}, \gamma}\right\|_{L^{2}}^{2}= d_{\gamma}(\omega),
\end{equation*}
it which is a contradiction.
\end{remark}

%\begin{figure}[h]
%\centering
%%\fbox{\includegraphics[height=32mm]{figura1}}
%%\caption{$\phi_{\omega, \gamma}$ as a function of $x$, $\omega=1$ and $\gamma=-1$.}
%\label{caja}
%\end{figure}

In order to prove Theorem \ref{ESSW} we need several preliminary lemmas.  In the first lemma, we recall the logarithmic Sobolev inequality. For a proof we refer to \cite[Theorem 8.14]{ELL}.
\begin{lemma} \label{L1}
Let $f$ be any function in $ H^{1}(\mathbb{R})\setminus\{0\}$ and $\alpha$ be any positive number. Then,
\begin{equation*}
\int_{\mathbb{R}}\left|f(x)\right|^{2}\mathrm{Log}\left|f(x)\right|^{2}dx\leq \frac{\alpha^{2}}{\pi} \|f' \|^{2}_{L^{2}}+\left(\mathrm{Log} \|f \|^{2}_{L^{2}}-\left(1+\mathrm{Log}\,\alpha\right)\right)\|f \|^{2}_{L^{2}}.
\end{equation*}
\end{lemma}
\begin{lemma}\label{L2}
Let $\gamma>0$ and $\omega\in \mathbb{R}$. Then, the quantity $d_{\gamma}(\omega)$ is positive and satisfies
\begin{equation}\label{EA}
d_{\gamma}(\omega)\geq \sqrt{{\pi}/{8}}\, e^{\omega+1}e^{-\frac{\gamma ^{2}}{2}}.
\end{equation}
\end{lemma}
\begin{proof}
Let $u\in W(\mathbb{R}^{})  \setminus  \left\{0 \right\}$ be such that  $I_{\omega,\gamma}(u)=0$. By H\"{o}lder inequality  we have  
\begin{equation}\label{EA2}
2\gamma\left|u(0)\right|^{2}\leq {\gamma^{2}}\left\|u\right\|_{L^{2}}^{2}+\left\|\partial_{x}u\right\|_{L^{2}}^{2}.
\end{equation}
Moreover, using \eqref{EA2},  the logarithmic Sobolev inequality with $\alpha=\sqrt{\frac{\pi}{2}}$ and $I_{\omega,\gamma}(u)=0$, we obtain
\begin{equation*}
 \left(\omega-\frac{\gamma ^{2}}{2} +1+\mbox{Log}\left(\sqrt{\frac{\pi}{2}}\right)\right)\left\|u\right\|^{2}_{L^{2}}\leq \left(\mbox{Log}\left\|u\right\|^{2}_{L^{2}} \right)\left\|u\right\|^{2}_{L^{2}},
\end{equation*}
which implies that 
\begin{equation*}
\left\|u\right\|^{2}_{L^{2}}\geq \sqrt{{\pi}/{2}}\, e^{\omega+1}e^{-\frac{\gamma ^{2}}{2}}.
\end{equation*}
Finally, by the definition of $d_{\gamma}(\omega)$ given in \eqref{MPE}, we get \eqref{EA}.
\end{proof}
\begin{lemma} \label{L3}
Let $\gamma>0$. The following  inequality holds for any $\omega\in \mathbb{R}$:
\begin{equation}\label{EIN}
 d_{\gamma}(\omega)<d_{0}(\omega).
\end{equation} 
\end{lemma}
\begin{proof}
We first remark that by \cite[Remark II.3]{CALO} the profile standing-wave $\phi_{\omega, 0}$ is a  minimizer of 
\begin{equation*}
d_{0}(\omega)=\mbox{inf}\left\{S_{\omega,0}(u):u\in W(\mathbb{R}^{}) \setminus  \left\{0 \right\},  I_{\omega,0}(u)=0\right\},
\end{equation*} 
that is, $S_{\omega,0}(\phi_{\omega,0})=d_{0}(\omega)$ and $ I_{\omega,0}(\phi_{\omega,0})=0$. On the other hand, easy computations permit us to obtain 
$$
I_{\omega,\gamma}(\phi_{\omega,0})=I_{\omega,0}(\phi_{\omega,0})-\gamma\left|\phi_{\omega_{}, 0}(0)\right|^{2}=-\gamma e^{\left(\omega+1\right)}<0.
$$
  Thus, there exist  $0<\lambda<1$ such that $I_{\omega,\gamma}(\lambda\phi_{\omega,0})=0$. Therefore, by the definition of $d_{\gamma}(\omega)$, we see that 
\begin{equation*}
d_{\gamma}(\omega)\leq S_{\omega,\gamma}(\lambda\phi_{\omega,0})=\lambda^{2}S_{\omega,0}(\phi_{\omega,0})<S_{\omega,0}(\phi_{\omega,0})=d_{0}(\omega),
\end{equation*}
and the proof of the Lemma is finished.
\end{proof}

\begin{remark} When $\gamma<0$ we have $d_{\gamma}(\omega)=d_{0}(\omega)$ for any $\omega\in\mathbb{R}$. Indeed, let $u\in W(\mathbb{R}) \setminus  \left\{0 \right\}$ be such that  $I_{\omega,0}(u)=0$. By direct computations, we see that  $I_{\omega,\gamma}(u)=-\gamma\left|u(0)\right|^{2}>0$. Then, there is  $s>1$ such that $I_{\omega,\gamma}(su)=0$. Then, since $S_{\omega,\gamma}(su)=s^2 S_{\omega, 0}(u)\geqq S_{\omega, 0}(u)\geqq d_{0}(\omega)$, we obtain from the definition of $d_{\gamma}(\omega)$ given in \eqref{MPE}, that $d_{0}(\omega)\leq d_{\gamma}(\omega)$. On the other hand, we define $w_{n}(x)=\phi_{\omega,0}(x-n)$ for  $n\in\mathbb{N}$. It is clear that $d_{0}(\omega)=S_{\omega,0}(w_{n})$ and $I_{\omega,\gamma}(w_{n})=-\gamma\left|\phi_{\omega,0}(n)\right|^{2}>0$. Thus, there exist $\lambda_{n}>1$ such that $I_{\omega,\gamma}(\lambda_{n} w_{n})=0$ for any $n\in\mathbb{N}$ and  $\lim_{n\rightarrow \infty}\lambda_{n}=1$. Then, by the definition of $d_{\gamma}(\omega)$ and from $I_{\omega,0}(w_n)=0$ we obtain
\begin{equation*}
d_{\gamma}(\omega)\leq S_{\omega,\gamma}(\lambda_{n}w_{n})=\lambda_{n}^{2}S_{\omega,0}(w_{n})=\lambda_{n}^{2}d_{0}(\omega),
\end{equation*}
which implies that  $d_{\gamma}(\omega)\leq d_{0}(\omega)$ and thus $d_{\gamma}(\omega)=d_{0}(\omega)$.
\end{remark}

The following lemma is a variant of the Br\'ezis-Lieb lemma from \cite{LBL}.

\begin{lemma} \label{L4}
Let  $\left\{u_{n}\right\}$ be a bounded sequence in $W(\mathbb{R})$ such that $u_{n}\rightarrow u$ a.e. in $\mathbb{R}$. Then $u\in W(\mathbb{R})$ and 
\begin{equation*}
\lim_{n\rightarrow \infty}\int_{\mathbb{R}}\left\{\left|u_{n}\right|^{2}\mathrm{Log}\left|u_{n}\right|^{2}-\left|u_{n}-u\right|^{2}\mathrm{Log}\left|u_{n}-u\right|^{2}\right\}dx=\int_{\mathbb{R}}\left|u\right|^{2}\mathrm{Log}\left|u\right|^{2}dx.
\end{equation*}
\end{lemma}
\begin{proof}
We first recall that, by \eqref{IFD}-Appendix,  $\left|z\right|^{2}\mbox{Log}\left|z\right|^{2}=A(\left|z\right|)-B(\left|z\right|)$  for every  $z\in\mathbb{C}$.  By the weak-lower semicontinuity of the ${L^{2}(\mathbb{R})}$-norm and Fatou lemma we have $u\in W(\mathbb{R})$.  It is clear that the sequence $\left\{u_{n}\right\}$ is  bounded in ${L^{A}(\mathbb{R})}$. Since $A$ in (\ref{IFD}) is convex and increasing function with $A(0)=0$, it is follows from Br\'ezis-Lieb lemma \cite[Theorem 2 and Examples (b)]{LBL} that 
\begin{equation}\label{AC}
\lim_{n\rightarrow \infty}\int_{\mathbb{R}} \left| A(\left|u_{n}\right|)-A(\left|u_{n}-u\right|)- A(\left|u_{}\right|)\right|dx=0.
\end{equation}
On the other hand, by the continuous embedding $W(\mathbb{R})\hookrightarrow H^{1}(\mathbb{R})$, we have  that  $\left\{u_{n}\right\}$ is also bounded in ${H^{1}(\mathbb{R})}$. An easy calculation shows that the function $B$ defined  in (\ref{IFD}) is convex, increasing and  nonnegative with    
$B(0)=0$. Furthermore, by H\"{o}lder and Sobolev inequalities, for any $u$, $v\in H^{1}(\mathbb{R})$  we have the following key inequality,
\begin{equation}\label{DB}
\int_{\mathbb{R}}\left|B(\left|u(x)\right|)- B(\left|v(x)\right|)\right|dx\leq C\left(1+ \left\|u\right\|^{2}_{H^{1}(\mathbb{R})}+ \left\|v\right\|^{2}_{H^{1}(\mathbb{R})} \right)\left\|u-v\right\|_{{L^{2}}}.
\end{equation}
Then, the function $B$ satisfies the hypotheses of \cite[Theorem 2 and Examples (b)]{LBL} and therefore
\begin{equation}\label{BC}
\lim_{n\rightarrow \infty}\int_{\mathbb{R}} \left| B(\left|u_{n}\right|)-B(\left|u_{n}-u\right|)- B(\left|u_{}\right|)\right|dx=0.
\end{equation}
Thus the result follows from \eqref{AC} and  \eqref{BC}. 
\end{proof}

\begin{proof}[ {\bf{Proof of Theorem \ref{ESSW}}}] We use the argument in \cite[Proposition 3]{FO}(see also \cite{ADNP}).  Let $\left\{ u_{n}\right\} \subseteq W(\mathbb{R}^{})$ be a minimizing sequence for $d_{\gamma}(\omega)$, then the sequence $\left\{ u_{n}\right\}$ is bounded in  $W(\mathbb{R}^{})$. Indeed, it is clear that the sequence $\|u_{n}\|^{2}_{L^{2}}$ is bounded. Moreover, using \eqref{EA2}, the logarithmic Sobolev inequality and recalling that $I_{\omega,\gamma}(u_{n})=0$, we obtain
\begin{equation*}
\left(\frac{1}{2}-\frac{\alpha^{2}}{\pi}\right)\left\|u'_{n}\right\|^{2}_{L^{2}}\leq \mbox{Log}\Biggl(\frac{e^{\frac{\gamma^{2}}{2}}e^{-\left(\omega+1\right)}}{\alpha^{}}\Biggr)\left\|u_{n}\right\|^{2}_{L^{2}} +\left(\mbox{Log}\left\|u_{n}\right\|^{2}_{L^{2}} \right)\left\|u_{n}\right\|^{2}_{L^{2}}.
\end{equation*}
Taking $\alpha>0$ sufficiently small, we see that $\|u'_{n}\|^{2}_{L^{2}}$ is bounded, so the sequence $\left\{ u_{n}\right\}$ is bounded in $H^{1}(\mathbb{R}^{})$.  Then, using  $I_{\omega,\gamma}(u_{n})=0$ again, and \eqref{DB} we obtain
\begin{equation*}
\left\| u'_{n}\right\|^{2}_{L^{2}}+\int_{\mathbb{R}^{}}A\left(\left|u_{n}(x)\right|\right)dx\leq C,
\end{equation*}
which implies, by \eqref{DA1} in the Appendix, that the sequence $\left\{ u_{n}\right\}$ is bounded in $W^{}(\mathbb{R}^{})$. Furthermore, since $W^{}(\mathbb{R}^{})$ is a reflexive Banach space, there is $\varphi \in W(\mathbb{R}^{})$ such that, up to a subsequence, $u_{n}\rightharpoonup \varphi$ weakly in $W^{}(\mathbb{R}^{})$ and $u_{n}(x)\rightarrow \varphi(x)$ $a.e.$  $x\in\mathbb{R}$.

Next, we show that $\varphi$ is nontrivial. Suppose, by contradiction, that $\varphi\equiv 0$. Since the embedding $H^{1}\left(-1,1\right)\hookrightarrow C\left[-1,1\right]$ is compact, we see that $u_{n}(0)\rightarrow \varphi(0)=0$. Thus, since  $I_{\omega,\gamma}(u_{n})=0$ we obtain
\begin{equation}\label{AA}
\lim_{n \to \infty} I_{\omega,0}(u_{n})=\gamma \lim_{n \to \infty} \left|u_{n}(0)\right|^{2}= 0.
\end{equation}
Define the sequence $v_{n}(x)=\lambda_{n}u_{n}(x)$ with 
\begin{equation*}
\lambda_{n}=\exp\left(\frac{I_{\omega,0}(u_{n})}{2\|u_{n}\|^{2}_{L^{2}}}\right),
\end{equation*}
where $\exp(x)$ represents the exponential function. Then, it follows from  \eqref{AA} that $\lim_{n\rightarrow \infty}\lambda_{n}=1$. Moreover,  an easy calculation shows that $I_{\omega,0}(v_{n})=0$ for any $n\in \mathbb{N}$.  Thus, by  the definition of $d_{\gamma}(\omega)$, it follows
\begin{equation*}
d_{0}(\omega)\leq \frac{1}{2} \lim_{n \to \infty}\|v_{n}\|^{2}_{L^{2}}= \frac{1}{2} \lim_{n \to \infty}\left\{\lambda^{2}_{n} \|u_{n}\|^{2}_{L^{2}}\right\}= d_{\gamma}(\omega),
\end{equation*}
that it is contrary to \eqref{EIN} and therefore we conclude that $\varphi$ is nontrivial.

Now we prove that $I_{\omega,\gamma}(\varphi)=0$. First,  assume by contradiction that $I_{\omega,\gamma}(\varphi)<0$. By elementary computations, we can see that there is $0<\lambda<1$ such that $I_{\omega,\gamma}(\lambda \varphi)=0$. Then, from the definition of $d_{\gamma}(\omega)$ and  the weak lower semicontinuity of the $L^{2}(\mathbb{R}^{})$-norm, we have
\begin{equation*}
d_{\gamma}(\omega)\leq \frac{1}{2}\left\|\lambda \varphi\right\|^{2}_{L^{2}}<\frac{1}{2}\left\|\varphi\right\|^{2}_{L^{2}}\leq \frac{1}{2}\liminf\limits_{n\rightarrow \infty}\left\|u_{n}\right\|^{2}_{L^{2}}=d_{\gamma}(\omega),
\end{equation*}
it which is impossible. On the other hand, assume that $I_{\omega,\gamma}(\varphi)>0$. Since the embedding  $W(\mathbb{R}^{})\hookrightarrow {H^{1}(\mathbb{R}^{})}$ is continuous, we see that $u_{n}\rightharpoonup \varphi$ weakly in $H^{1}(\mathbb{R})$. Thus, we have  
\begin{align}
& \left\|u_{n}\right\|^{2}_{L^{2}}-\left\|u_{n}-\varphi\right\|^{2}_{L^{2}}-\left\|\varphi\right\|^{2}_{L^{2}}\rightarrow0 \label{2C11}\\
&\left\|u'_{n}\right\|^{2}_{L^{2}}-\left\|u'_{n}-\varphi'\right\|^{2}_{L^{2}}-\left\|\varphi' \right\|^{2}_{L^{2}}\rightarrow0, \label{2C12} 
\end{align}
as $n\rightarrow\infty$. Combining \eqref{2C11}, \eqref{2C12} and Lemma \ref{L4} leads to 
\begin{equation*}
\lim_{n\rightarrow \infty}I_{\omega, \gamma}(u_{n}-\varphi)=\lim_{n\rightarrow \infty}I_{\omega, \gamma}(u_{n})-I_{\omega,\gamma}(\varphi)=-I_{\omega,\gamma}(\varphi),
\end{equation*}
which combined with  $I_{\omega, \gamma}(\varphi)> 0$ give us  that $I_{\omega, \gamma}(u_{n}-\varphi)<0$ for sufficiently large $n$. Thus, by \eqref{2C11} and  applying the same argument as above, we see that 
\begin{equation*}
d_{\gamma}(\omega)\leq\frac{1}{2} \lim_{n\rightarrow \infty}\left\|u_{n}-\varphi\right\|^{2}_{L^{2}}=d_{\gamma}(\omega)-\frac{1}{2}\left\|\varphi\right\|^{2}_{L^{2}},
\end{equation*}
which is a contradiction because $\|\varphi\|^{2}_{L^{2}}>0$. Then, we deduce that $I_{\omega,\gamma}(\varphi)=0$. Finally, by the weak lower semicontinuity of the $L^{2}(\mathbb{R}^{})$-norm,  we have
\begin{equation}\label{inequa}
d_{\gamma}(\omega)\leq \frac{1}{2}\left\|\varphi\right\|^{2}_{L^{2}}\leq \frac{1}{2} \liminf\limits_{n\rightarrow \infty}\left\|u_{n}\right\|^{2}_{L^{2}}=d_{\gamma}(\omega),
\end{equation}
which implies, by the definition of $d_{\gamma}(\omega)$, that $\varphi\in \mathcal{N}_{\omega,\gamma}$. Moreover, by Remark \ref{RM} and  Proposition \ref{UDGS} there exist $\theta\in\mathbb{R}$ such that $\varphi(x)=e^{i\theta}\phi_{\omega,\gamma}(x)$.  This concludes the proof of Theorem \ref{ESSW}.
\end{proof}

\section{Stability of the ground states}
\label{S:4}

This section is devoted to the proof of Theorem \ref{2ESSW}. We first prove compactness of the minimizing sequences.
\begin{lemma} \label{CSM}
Let $\left\{ u_{n}\right\}\subseteq W(\mathbb{R}^{})$ be a minimizing sequence for $d_{\gamma}(\omega)$. Then, up to a subsequence,  there is  $\theta_{}\in \mathbb{R}$ such that $u_{n}\rightarrow e^{i\theta_{}}\phi_{\omega, \gamma}$ in $W(\mathbb{R}^{})$.
\end{lemma}
\begin{proof}
By Theorem \ref{ESSW}, we see that there is  $\varphi\in \mathcal{N}_{\omega,\gamma}$ such that, up to a subsequence, $u_{n}\rightharpoonup \varphi$ weakly  in $W(\mathbb{R}^{})$ and $u_{n}(x)\rightarrow \varphi(x)$ $a.e.$ $x\in \mathbb{R}$. Furthermore, by  (\ref{MPE}) and (\ref{inequa}) we have  $u_{n}\rightarrow \varphi$   in $L^{2}(\mathbb{R}^{})$. Then, since the sequence $\left\{ u_{n}\right\}$ is bounded in $H^{1}(\mathbb{R}^{})$, from \eqref{DB} we obtain
\begin{equation*}
 \lim_{n\rightarrow \infty}\int_{\mathbb{R}^{}}B\left(\left|u_{n}(x)\right|\right)dx=\int_{\mathbb{R}^{}}B\left(\left|\varphi(x)\right|\right)dx.
\end{equation*}
Thus,  since $I_{\omega,\gamma}(u_{n})=I_{\omega,\gamma}(\varphi)=0$  for any $n\in \mathbb{N}$,  we obtain
\begin{equation}\label{2BX1}
\lim_{n\rightarrow \infty}\left\{\left\|u'_{n}\right\|^{2}_{L^{2}}+\int_{\mathbb{R}^{}}A\left(\left|u_{n}(x)\right|\right)dx\right\}=\left\|\varphi'\right\|^{2}_{L^{2}}+\int_{\mathbb{R}^{}}A\left(\left|\varphi(x)\right|\right)dx.
\end{equation}
Moreover, by \eqref{2BX1},  the weak lower semicontinuity of the $L^{2}(\mathbb{R}^{})$-norm and Fatou lemma, we deduce (see, for example, \cite[Lemma 12 in chapter V]{AH1})
\begin{align}
& \lim_{n\rightarrow \infty}\left\|u'_{n}\right\|^{2}_{L^{2}}=\left\|\varphi'\right\|^{2}_{L^{2}},\label{N1}\\
& \lim_{n \to \infty}\int_{\mathbb{R}}A\left(\left|u_{n}(x)\right|\right)dx=\int_{\mathbb{R}}A\left(\left|\varphi(x)\right|\right)dx. \label{N2} \end{align}
Since $u_{n}\rightharpoonup \varphi$ weakly  in $H^{1}(\mathbb{R})$, it follows from \eqref{N1} that $u_{n}\rightarrow\varphi$  in $H^{1}(\mathbb{R})$.  Finally, by Proposition \ref{orlicz}-{\it ii)}  (Appendix below) and \eqref{N2} we have $u_{n}\rightarrow\varphi$  in $L^{A}(\mathbb{R})$. Thus, by definition of the $W(\mathbb{R})$-norm, we infer that $u_{n}\rightarrow\varphi$  in $W(\mathbb{R})$. Thus, by Remark \ref{RM} and  Proposition \ref{UDGS} there exist $\theta\in\mathbb{R}$ such that $\varphi(x)=e^{i\theta}\phi_{\omega,\gamma}(x)$. This finishes  the proof.
\end{proof}

\begin{proof}[ {\bf{Proof of Theorem \ref{2ESSW}}}] 
We argue by contradiction.  Suppose that $\phi_{\omega, \gamma}$ is not stable in $W(\mathbb{R})$. Then, there is $\epsilon>0$ and two  sequences $\left\{u_{n,0}\right\}\subset W(\mathbb{R})$, $\left\{t_{n}\right\}\subset \left(0,\,\infty\right)$ such that 
\begin{align}
&\left\|u_{n,0}-\phi_{\omega,\gamma}\right\|_{ W(\mathbb{R})}\rightarrow 0,\;\;\; \,\,\,\ \textrm{ as  }\,\, n\rightarrow \infty, \label{T21} \\
&{\rm\inf\limits_{\theta\in \mathbb{R}}} \|u_{n}(t_{n})-e^{i\theta}\phi_{\omega_{}, \gamma} \|_{W(\mathbb{R}^{})}\geq{\epsilon}, \;\;\;\,\,\ \textrm{for any $n\in \mathbb{R}$,}\label{T22}
\end{align}
where $u_{n}$ is the solution of \eqref{00NL} with initial data $u_{n,0}$ (see Proposition \ref{PCS}). Set $v_{n}(x)= u_{n}(t_{n},x)$. By \eqref{T21} and conservation laws, we obtain
\begin{gather}
\left\|v_{n}\right\|^{2}_{L^{2}}=\left\|u_{n}(t_{n})\right\|^{2}_{L^{2}}=\left\|u_{n,0}\right\|^{2}_{L^{2}}\rightarrow \left\|\phi_{\omega,\gamma}\right\|^{2}_{L^{2}}\label{CE1} \\
E(v_{n})=E(u_{n}(t_{n}))=E(u_{n,0})\rightarrow E(\phi_{\omega,\gamma}),
\label{CE2}
\end{gather}
as $n\rightarrow \infty$.  In particular, it follows from \eqref{CE1} and \eqref{CE2} that, as $n\rightarrow \infty$, 
\begin{equation}\label{A12}
S_{\omega,\gamma}(v_{n})\rightarrow S_{\omega,\gamma}(\phi_{\omega,\gamma})=d_{\gamma}(\omega).
\end{equation}
Next, by combining \eqref{CE1} and \eqref{A12} lead us to $I_{\omega,\gamma}(v_{n})\rightarrow 0$ as $n\rightarrow \infty$.
Define the sequence $f_{n}(x)=\rho_{n}v_{n}(x)$ with
\begin{equation*}
\rho_{n}=\exp\left(\frac{I_{\omega,\gamma}(v_{n})}{2\|v_{n}\|^{2}_{L^{2}}}\right).
\end{equation*}
 It is clear that $\lim_{n\rightarrow \infty}\rho_{n}=1$ and $I_{\omega,\gamma}(f_{n})=0$ for any $n\in\mathbb{R}$. Furthermore, since the sequence $\left\{v_{n}\right\}$  is bounded in $W(\mathbb{R})$, we get $\|v_{n}-f_{n}\|_{W(\mathbb{R})}\rightarrow 0$ as $n\rightarrow \infty$. Then, by  \eqref{A12}, we have that $\left\{f_{n}\right\}$ is a minimizing sequence for $d_{\gamma}(\omega)$. Thus, by Lemma \ref{CSM}, up to a subsequence, there is $\theta_{0}\in \mathbb{R}$  such that $f_{n}\rightarrow e^{i\theta_{0}}\phi_{\omega, \gamma}$ in  $W(\mathbb{R})$. Therefore, by using the triangular inequality, we have
\begin{equation*}
\|u_{n}(t_{n})-e^{i\theta_{0}}\phi_{\omega,\gamma}\|_{ W(\mathbb{R})}\leq \|v_{n}-f_{n}\|_{ W(\mathbb{R})}+\|f_{n}-e^{i\theta_{0}}\phi_{\omega, \gamma}\|_{ W(\mathbb{R})}\rightarrow 0,
\end{equation*}
as $n\rightarrow \infty$, it which is a contradiction with \eqref{T22}. This finishes the proof.
\end{proof}

\section*{Acknowledgements}

The research of J. Angulo Pava was supported by CNPq/Brazil,  Processo 312435/2015-0.   A. Hernandez Ardila was supported by CAPES and CNPq/Brazil. The results in this paper form a part of the second author's Ph.D. thesis.

\section{Appendix}
\label{S:5}

The functional of energy in  (\ref{energy}), in general,  fails to be finite and of class $C^{1}$ on $H^{1}(\mathbb{R^{}}^{})$. Due to this loss of smoothness, in order to study existence of solutions to (\ref{00NL}) and (\ref{eq2}), it is convenient  to work in a suitable Banach space endowed with a Luxemburg type norm in order to make functional $E$  well defined and $C^{1}$ smooth.  So, define 
\begin{equation*}
F(z)=\left|z\right|^{2}\mbox{Log}\left|z\right|^{2}\,\,\,\,\,\,\,\,  \text{for every  $z\in\mathbb{C}$},
\end{equation*}
and as in \cite{CL},  we define the functions  $A$, $B$ on $\left[0, \infty\right)$  by 
\begin{equation}\label{IFD}
A(s)=
\begin{cases}
-s^{2}\,\mbox{Log}(s^{2}), &\text{if $0\leq s\leq e^{-3}$;}\\
3s^{2}+4e^{-3}s^{}-e^{-6}, &\text{if $ s\geq e^{-3}$;}
\end{cases}
\,\,\,\,\,\,\,\,\,  B(s)=F(s)+A(s).
\end{equation}
Furthermore, let be functions $a$, $b$, defined by
\begin{equation}\label{abapex}
a(z)=\frac{z}{|z|^{2}}\,A(\left|z\right|)\,\, \text{ and  }\,\, b(z)=\frac{z}{|z|^{2}}\,B(\left|z\right|) \text{  for $z\in \mathbb{C}$, $z\neq 0$}.
\end{equation}
Notice that we have $b(z)-a(z)=z\mathrm{Log}\left|z\right|^{2}$. It follows that $A$ is a nonnegative  convex and increasing function, and $A\in C^{1}\left([0,+\infty)\right)\cap C^{2}\left((0,+\infty)\right)$. The Orlicz space $L^{A}(\mathbb{R})$ corresponding to $A$ is defined by
\begin{equation*}
L^{A}(\mathbb{R})=\left\{u\in L^{1}_{loc}(\mathbb{R}) : A(\left|u\right|)\in L^{1}_{}(\mathbb{R})\right\}, 
\end{equation*}
equipped with the Luxemburg norm 
\begin{equation*}
\left\|u\right\|_{L^{A}(\mathbb{R})}={\inf}\left\{k>0: \int_{\mathbb{R}^{}}A\left(k^{-1}{\left|u(x)\right|}\right)dx\leq 1 \right\}.
\end{equation*}
Here as usual $L^{1}_{loc}(\mathbb{R})$ is the space of all locally Lebesgue integrable functions. It is proved in \cite[Lemma 2.1]{CL} that $A$ is a Young-function which is $\Delta_{2}$-regular and $\left(L^{A}(\mathbb{R}),\|\cdot\|_{L^{A}} \right)$ is a separable reflexive  Banach space.\\

Next, we consider the reflexive Banach space $W(\mathbb{R})=H^{1}(\mathbb{R})\cap L^{A}(\mathbb{R})$ equipped with the usual norm $ \left\|u\right\|_{W(\mathbb{R}^{})}=\left\|u\right\|_{H^{1}(\mathbb{R}^{})}+\left\|u\right\|_{L^{A}(\mathbb{R}^{})}$. We can see that $W(\mathbb{R})=\bigl\{u\in H^{1}(\mathbb{R}):\left|u\right|^{2}\mathrm{Log}\left|u\right|^{2}\in L^{1}(\mathbb{R})\bigr\}$ (see (\ref{W})). This follows from the definition of the spaces $L^{A}(\mathbb{R})$ and $W(\mathbb{R})$ (see \cite[Proposition 2.2]{CL} for more details). Furthermore, one has the following chain of continuous embedding:
\begin{equation*}
W(\mathbb{R})\hookrightarrow L^{2}(\mathbb{R})\hookrightarrow W^{\prime}(\mathbb{R}),
\end{equation*}
where $W^{\prime}(\mathbb{R})=H^{ -1}(\mathbb R)+(L^A(\mathbb R))' $  is the dual space of $W(\mathbb{R})$ equipped with the usual norm.

Next, we list some properties of the Orlicz space $L^{A}(\mathbb{R})$ that we have used through our manuscript.  For a proof of such statements we refer to \cite[Lemma 2.1]{CL}.

\begin{proposition} \label{orlicz}
Let $\left\{u_{{m}}\right\}$ be a sequence in  $L^{A}(\mathbb{R})$, the following facts hold:\\
{\it i)} If  $u_{{m}}\rightarrow u$ in $L^{A}(\mathbb{R})$, then $A(\left|u_{{m}}\right|)\rightarrow A(\left|u\right|)$ in $L^{1}(\mathbb{R})$ as   $n\rightarrow \infty$.\\
{\it ii)} Let  $u\in L^{A}(\mathbb{R})$. If  $u_{m}(x)\rightarrow u(x)$ $a.e.$  $x\in\mathbb{R}$ and if 
\begin{equation*}
\lim_{n \to \infty}\int_{\mathbb{R}}A\left(\left|u_{m}(x)\right|\right)dx=\int_{\mathbb{R}}A\left(\left|u(x)\right|\right)dx,
\end{equation*}
then $u_{{m}}\rightarrow u$ in $L^{A}(\mathbb{R})$ as   $n\rightarrow \infty$.\\
{\it iii)} For any $u\in L^{A}(\mathbb{R})$, we have
\begin{equation}\label{DA1}
{\rm min} \left\{\left\|u\right\|_{L^{A}(\mathbb{R})},\left\|u\right\|^{2}_{L^{A}(\mathbb{R})}\right\}\leq  \int_{\mathbb{R}}A\left(\left|u(x)\right|\right)dx\leq {\rm max} \left\{\left\|u\right\|_{L^{A}(\mathbb{R})},\left\|u\right\|^{2}_{L^{A}(\mathbb{R})}\right\}.
\end{equation}
\end{proposition}

The following Lemma is the base for showing the $C^1$-property of the energy functional  $E$ in (\ref{energy}) on $W(\mathbb R)$. 

\begin{lemma} \label{APEX23}
The operator $L: u\rightarrow \partial^{2}_{x}u+ \gamma\delta(x)u+u\,  \mathrm{Log}\left|u\right|^{2}$ is continuous from  $W(\mathbb{R})$  to $W^{\prime}(\mathbb{R})$. The image under $L$ of a bounded subset of $W(\mathbb{R})$ is a bounded subset of $W^{\prime}(\mathbb{R})$.
\end{lemma}
\begin{proof}
As usual, the operator $\left(-\partial^{2}_{x}-\gamma\delta(x)\right)u$ is naturally extended to $H^{1}(\mathbb{R})\rightarrow H^{-1}(\mathbb{R})$ via the relation (see \eqref{UKI})
\begin{equation*}
\left\langle (-\partial^{2}_{x}-\gamma\delta(x))u,v\right\rangle=\mathfrak{t_{\gamma}}(u_{},v),  \quad  \textrm{for} \quad u,v\in H^{1}(\mathbb{R}).
\end{equation*}
Now, using ${W}({\mathbb{R}})\hookrightarrow H^{1}(\mathbb{R})$, we obtain that the linear operator $u\rightarrow -\partial^{2}_{x}u-\gamma\delta(x)u$ is continuous from ${{W}}({\mathbb{R}})$ to ${W}^{\prime}({\mathbb{R}})$. Thus, since $u\rightarrow u\mathrm{Log}\left|u\right|^{2}$ is continuous and bounded  from ${{W}}({\mathbb{R}})$ to ${W}^{\prime}({\mathbb{R}})$ (see  \cite[Lemma 2.6]{CL}), it follows that the operator $L: {{W}}({\mathbb{R}})\rightarrow {W}^{\prime}({\mathbb{R}})$ is continuous and bounded. Lemma \ref{APEX23} is thus proved.
\end{proof}

From Lemma \ref{APEX23},  we have the following consequence:

\begin{proposition} \label{DFFE}
The operator $E: W(\mathbb{R})\rightarrow \mathbb R$  is of class $C^{1}$ and  for $u\in W(\mathbb{R})$ the  Fr\'echet derivative of $E$ in $u$ exists and it is given by  
\begin{equation*}
E^{\prime}(u)=-\partial^{2}_{x} u- \gamma\delta(x)u-u\, \mathrm{Log}\left|u\right|^{2}-u \in W'(\mathbb R)
\end{equation*}
\end{proposition}
\begin{proof}
We first show that $E$  is continuous. Notice that 
\begin{equation}\label{CCC}
E(u)=\frac{1}{2}\mathfrak{t}_{\gamma}(u)+\frac{1}{2}\int_{\mathbb{R}}A(\left|u_{}\right|)dx-\frac{1}{2}\int_{\mathbb{R}}B(\left|u_{}\right|)dx.
\end{equation}
The first term in the right-hand side of \eqref{CCC} is continuous of $H^{1}(\mathbb{R})\rightarrow \mathbb R$, and it follows from 
Proposition \ref{orlicz}(i) that the second term is continuous of $L^{A}(\mathbb{R})\rightarrow \mathbb{R}$. Moreover, by \eqref{DB} we get that the third term  in the right-hand side of \eqref{CCC} is continuous of $H^{1}(\mathbb{R})\rightarrow \mathbb R$. Therefore, $E\in C(W(\mathbb{R}),\mathbb{R})$. Now, direct calculations show that, for $u$, $v\in W(\mathbb{R})$, $t\in (-1,1)$ (see \cite[Proposition 2.7]{CL}),
\begin{equation*}
\lim_{t\rightarrow 0} \frac{E(u+tv)-E(u)}{t}=\left\langle -\partial^{2}_{x} u- \gamma\delta(x)u-u\, \mbox{Log}\left|u\right|^{2}-u,v\right\rangle_{W(\mathbb{R})^{}-W^{\prime}(\mathbb{R})}.
\end{equation*}
Thus, $E$ is G\^ateaux differentiable. Then, by Lemma \ref{APEX23} we see that $E$ is  Fr\'echet differentiable  and $E^{\prime}(u)=-\partial^{2}_{x} u- \gamma\delta(x)u-u\, \mathrm{Log}\left|u\right|^{2}-u$.
\end{proof}

%in the sense that for $h\in W(\mathbb{R})$,
%$$
%E'(u)(h)= \langle E'(u), h\rangle= \Re\Big [ \int_{\mathbb R} u'\overline{h'}\,dx - \int_{\mathbb R} u\overline{h} \mathrm{Log} |u|^2\,dx -  %\int_{\mathbb R} u\overline{h}\,dx\Big ]-\gamma \Re (u(0)\overline{h(0)}).
%$$

\end{document}